\UseRawInputEncoding
\documentclass[11pt]{article}

\usepackage{mathtools,amsthm,verbatim,amssymb,amscd,graphicx,graphics,nccmath,braket,tipa,tikz-cd,mathdots,nicematrix,setspace,dsfont,xcolor,url,caption,float,footmisc,subcaption,ragged2e}

\usepackage[mathscr]{euscript}
\usepackage[scr=boondox]{mathalpha}

\usepackage[shortlabels]{enumitem}
\usepackage[affil-it]{authblk}

\usepackage[utf8]{inputenc}    
\usepackage[autostyle]{csquotes}

\usepackage[
    backend=biber,
    sorting=nty,
    citestyle=numeric-comp,
    bibstyle=ieee,
    natbib=true,
    url=true, 
    doi=false,
    isbn=false,
    eprint=false
]{biblatex}
\usepackage[]{hyperref}
\hypersetup{
  colorlinks   = true, 
  urlcolor     = blue, 
  linkcolor    = blue, 
  citecolor   = red 
}

\makeatletter
\g@addto@macro\th@plain{\thm@headpunct{}}
\makeatother

\theoremstyle{plain}
\newtheorem{theorem}{Theorem}[section]
\newtheorem{corollary}[theorem]{Corollary}
\newtheorem{proposition}[theorem]{Proposition}
\newtheorem{lemma}[theorem]{Lemma}

\newtheorem{conjecture}[theorem]{Conjecture}

\newtheorem{definition}[theorem]{Definition}
\newtheorem{note}[theorem]{Note}

\DeclarePairedDelimiter\abs{\lvert}{\rvert}
\DeclarePairedDelimiter\norm{\lVert}{\rVert}

\makeatletter
\let\oldabs\abs
\def\abs{\@ifstar{\oldabs}{\oldabs*}}

\let\oldnorm\norm
\def\norm{\@ifstar{\oldnorm}{\oldnorm*}}
\makeatother

\makeatletter
\newcommand*\bigcdot{\mathpalette\bigcdot@{.5}}
\newcommand*\bigcdot@[2]{\mathbin{\vcenter{\hbox{\scalebox{#2}{$\m@th#1\bullet$}}}}}
\makeatother

\providecommand{\keywords}[1]{\textbf{\textit{Keywords:}} #1}
\providecommand{\subjclass}[1]{\textbf{\textit{MSC Codes:}} #1}

\DeclareMathOperator{\diag}{diag}
\DeclareMathOperator{\tr}{tr}

\begin{document}

\title{Zeroing Diagonals, Conjugate Hollowization, and Characterizing Nondefinite Operators}
\author{David~R. Nicholus\thanks{drnicholus@gmail.com}}

\affil{Chicago Transit Authority\\
  William Rainey Harper College, Department of Mathematics}
  
\maketitle

\begin{abstract}
We prove the conjecture by Damm and Fa\ss bender that, for real traceless matrices $L,M$, there exists orthogonal $R$ such that $\diag(R^\top L R) = (0,...,0,0,0)$ and $\diag(R M R^\top) = (0,...,0,*,*)$. We also prove for any pair $L,M$ of complex Hermitian traceless matrices, there exists a unitary~$U$ such that $\diag(U^* L U) =\diag(U M U^*) = (0,...,0)$. The claims comprise a corollary to our more general theorem for $L,M$ of arbitrary trace. We also discuss severe limitations upon generalizing our theorem to general complex $L,M$.

By setting $L = M$, much is revealed concerning freedom and constraint involved in introducing 0s to the diagonal of a single operator. From this we prove a novel characterization of real traceless matrices and complex Hermitian traceless matrices, strengthening the seminal theorem by Fillmore that every complex square matrix is unitarily similar to a hollow matrix.

Our results are contextualized in a characterization of nondefinite matrices as a more general environment for introducing 0s to the main diagonal.
\end{abstract}

\\



\section{Introduction} \label{Sect:intro} 


A square matrix whose main diagonal consists only of 0s is a \emph{hollow matrix}. \cite{zC13,mF15,hK16,dN23,tD20,aN18,rB19,jG17} For example, all antisymmetric matrices are hollow, all graphs without loops have hollow adjacency matrices (so all tournament matrices are hollow), all conference matrices are hollow, and all traceless antidiagonal matrices are hollow. A similarity decomposition transforming a matrix into a hollow form is a \emph{hollowization}. \cite{tD20,aN18,dN23} Hollowization can be thought of as complementary to diagonalization, as hollow forms are the complement to diagonal forms. Early research focusing on hollow matrices and transformations into them was published by Horn and Schur. \cite{aH54,iS23} Some of the most notable results concerning transformations into hollow forms were published later by Fillmore in \cite{pF69} implying every traceless complex square matrix is unitarily similar to a hollow matrix, and we derive stronger versions of them here. Both pure and applied interest in hollow matrices and hollowization has led to an influx of research in recent years. \cite{zC13,mF15,hK16,dN23,aN18,tD20,rB19} Symmetric hollow matrices of various forms have been studied in \cite{zC13,mF15,hK16}, research regarding hollow orthogonal matrices is provided in \cite{rB19} continuing from \cite{jG67,pD71}, and various hollowizations to quasidiagonal forms are reviewed in \cite{dN23}. In \cite{aN18}, Neven and Bastin use related results from \cite{rT79} to prove the general separability problem of mixed states in quantum mechanics reduces to determining if a set of symmetric matrices representing the states is simultaneously unitarily hollowizable.

The importance of introducing 0s into matrices and vectors to numerical linear algebra and matrix analysis, especially below and above main diagonals, cannot be overstated. Indeed, Trefethen and Bau note ``The algorithms of numerical linear algebra are mainly built upon one technique used over and over again: putting zeros into matrices.'' \cite{lT97} Nonetheless, comparatively little research has explored matrix-theoretic hollowization. Indeed, of and in \cite{tD20}, Damm and Fa\ss bender note ``to the best of our knowledge, the current note is the first to treat hollowization problems from the matrix theoretic side''. Part of the purpose of this paper is to fill this void.

Whereas an eigendecomposition reveals a basis of orthogonal eigenvectors, a hollowization reveals a basis of orthogonal \emph{neutral vectors} -- vectors for which the quadratic form is 0. This is one of the reasons hollow matrices and hollowizations are useful in asymptotic eigenvalue determination and stabilization of linear systems. In \cite{tD20}, Damm and Fa\ss bender use hollowization to prove results on stabilization of linear systems by rotational forces or by noise, which themselves have seen recent utility for stochastic partial differential equations and Hamiltonian systems. \cite{nN00,rB86,tK19} The authors prove, for any pair $L,M$ of real traceless matrices, there exists an orthogonal~$R$ such that $R^\top L \, R$ is hollow and $R^\top M \, R$ is almost hollow. This form of hollowization along with related results is used first to provide a constructive proof of the theorem by Brickman in \cite{lB61} that the real joint numerical range of a pair of matrices is convex, then to prove every real traceless matrix of even size is hollowizable by an operator that is orthogonal and symplectic, then to develop a survey of results on unitary hollowization and simultaneous unitary hollowization, and finally to prove theorems on the stabilization of linear systems.

Most notably in \cite{tD20}, Damm and Fa\ss bender pose their Conjecture 22 that for every pair $L,M$ of real traceless matrices, there exists an orthogonal $R$ such that $R^\top L \, R$ is hollow and $R M R^\top$ is almost hollow.\footnote{Notice the transformation applied to $M$ is the inverse of that applied to~$L$.} Our paper proves their conjecture, as well as a stronger equivalent for complex Hermitian matrices, in Corollary \ref{C:conject22}.

Our results are expressed in the more general context of \emph{zeroing diagonals} -- the process of introducing zeros to main diagonals via similarity transformation; after discussing definitions, preliminary concepts, and conventions in Section \ref{Sect:prelim}, our paper introduces the proper context for this in Section \ref{Sect:nondefZero}, which is that of \emph{nondefinite operators} -- operators that are neither positive definite nor negative definite. Roughly speaking, zeroing diagonals is to nondefinite operators what hollowization is to traceless operators. Indeed, Proposition \ref{P:zdndEquiv} shows an operator being nondefinite, containing 0 in its numerical range, and being able to have any diagonal element be unitarily zeroed are all equivalent. We finish the section discussing some considerations on \emph{conjugate zeroing}, whereby a nonsingular operator $V$ introduces diagonal zeros to a pair of operators $L,M$ via $(V^{-1} L V,V M V^{-1})$.

In Section \ref{Sect:holTraceless} we review Fillmore's theorem that a complex matrix $M$ can have all but one diagonal element unitarily zeroed if and only if $\tr(M) \in W(M)$, as it and its proof are useful for claims that follow. In Section \ref{Sect:conjHol} we review claims concerning joint numerical ranges and prove three lemmas regarding conjugate zeroing as stairs to prove our main theorem. Lemma \ref{L:jointSpecMeans} generalizes Proposition~5 from \cite{tD20}, Lemma \ref{L:hHComp} reveals how Householder matrices will be critical to proving our main theorem, and Lemma \ref{L:conZer} gives a first glimpse of our conjugate hollowization scheme as well as base cases for our main theorem. We then prove the culmination of our work, Theorem~\ref{T:conHol}; letting $L,M \in \mathbb{C}^{n \times n}$, $l = \frac{\tr L}{n}$, and $m = \frac{\tr M}{n}$, if $L,M$ are real, then there exists orthogonal $R$ such that $\diag(R^\top L\, R) = (l,...,l,l,l)$ and $\diag(R \, M \,R^\top) = (m,...,m,*,*)$, and if $L,M$ are Hermitian, then there exists unitary $U$ such that $\diag(U^* L\, U) = (l,...,l)$ and $\diag(U M U^*) = (m,...,m)$. Conjecture~22 from \cite{tD20} and its stronger Hermitian equivalent follow in Corollary \ref{C:conject22} from $\tr(L) = \tr(M) = 0$. Note \ref{N:reqAss} discusses why a complete conjugate hollowization is not, in general, possible in the real case and why the Hermitian assumption cannot, in general, be weakened in the complex case. We end with Corollary \ref{C:conHolOne} where we derive novel characterizations of real traceless matrices and of Hermitian traceless matrices, acting as strengthened forms of some of Fillmore's claims from \cite{pF69}. Finally, in Section \ref{Sect:conc}, we summarize our results, suggest directions for future research into more general contexts for introducing 0s to diagonals, and submit our own Conjecture \ref{Con:extConjHol}.

Unless specified otherwise, matrices in this paper are complex square matrices, linear transformations are linear operators (isomorphisms), and similarity transformations are unitary.\footnote{Of course, ``complex'' includes ``real'' and ``unitary'' includes ``orthogonal'' as subcases, and many of our results pertain specifically to these subcases.}  A matrix or operator of \emph{size} $n$ refers to an $n \times n$ matrix. It should be understood zeroing refers to unitarily zeroing diagonal elements, and hollowization refers to unitary hollowization. Some of the claims proven in this paper can be extended to operators over any field of characteristic 0 and can be extended to nonsquare matrices in some way. We leave a detailed exploration for future research.

The diagonal always refers to the main diagonal, the identity matrix of size $n$ is denoted~$I_n$, and ``conjugation'' by a nonsingular matrix refers to a similarity transformation.

\section{Preliminaries and Definitions} \label{Sect:prelim} 

We first introduce some useful definitions and concepts.

\begin{definition} [Hollow, Almost Hollow, Zeroing, Hollowization] \label{D:holDefs}
	Let $V,M$ be matrices of size $n$ where $V$ is nonsingular.
	\begin{enumerate}[(i)]
	\item $M$ is {\bfseries hollow} if and only if every diagonal element of $M$ is 0. \cite{zC13,mF15,hK16,dN23,tD20,aN18,rB19,jG17}
	\item $M$ is {\bfseries almost hollow} if and only if $M$ is traceless and every diagonal element of $M$ is 0 except, at most, the last two. \cite{tD20,dN23}
	\item {\bfseries Zeroing} the diagonal or diagonal elements refers to introducing 0s to the diagonal via similarity transformation. 
	\item Similarity transformation $\mathcal{M} = V^{-1} M V$ is a {\bfseries hollowization} of $M$ with (nonsingular) \emph{hollowizer} $V$ if and only if $\mathcal{M}$ is hollow. If such a transformation exists, $M$ is \emph{hollowizable}. If $V$ is unitary/orthogonal, $V$ is a \emph{unitary/orthogonal hollowizer} for the \emph{unitary/orthogonal hollowization} of $M$. \cite{dN23,tD20,aN18} 
	\end{enumerate}
\end{definition}

A set of matrices is \emph{simultaneously hollowizable} if and only if each matrix in the set can be hollowized by the same matrix. \cite{dN23,tD20,aN18} 

The following definition is justified by Conjecture 22 from \cite{tD20} and our proof of it in Corollary~\ref{C:conject22}$(a)$. It is also useful for Corollary~\ref{C:conject22}$(b)$.

\begin{definition} [Conjugate Zeroing and Hollowization] \label{D:twiHollowize}
	Introducing 0s to the diagonals of $L,M$ via the similarity transformations $(\mathcal{L}, \mathcal{M}) = (V^{-1} L V,V M V^{-1})$ with nonsingular $V$ is a {\bfseries conjugate zeroing} of diagonal elements of $L,M$, and the conjugate zeroing is a {\bfseries conjugate hollowization} of $L,M$ if and only if $\mathcal{L},\mathcal{M}$ are hollow.	
\end{definition}

The subtle but important difference between a simultaneous zeroing and a conjugate zeroing is that the order of the zeroing operator $V$ is flipped in the latter, so the transformations applied to $L,M$ are inverses of each other. That is,

\indent \indent $(V^{-1} L V,V^{-1} M V)$ is a simultaneous zeroing, whereas 

\indent \indent $(V^{-1} L V,V M V^{-1})$ is a conjugate zeroing.

By setting $M = I$, we can see both simultaneous zeroing and conjugate zeroing are generalizations of zeroing. However, unlike simultaneous zeroing, the case $L = M$ is nontrivial for conjugate zeroing, and the definitions for conjugate zeroing and conjugate hollowization of a single operator follow. This case is explored at the end of Section \ref{Sect:conjHol}.

The zeroing and hollowization operations considered in this paper are unitary or orthogonal, so we use $U$ and $R$, respectively, to denote them, and the conjugate transpose and transpose, respectively, to indicate inverses.


\section{Zeroing Diagonal Elements and Nondefinite Operators} \label{Sect:nondefZero} 

\subsection{Invariances and Equivalences for Zeroing Diagonal Elements} \label{Subsect:invEquiv} 

Let $M,V$ be two matrices of the same size where $V$ is nonsingular, and let $c$ be a scalar. First, note zeroing diagonal elements is scale-invariant, so that zeroing diagonal elements of $M$ is equivalent to zeroing diagonal elements of $cM$ for nonzero $c$. A note generalizing Remark 2.4$(c)$ from \cite{tD20} will also be useful.

\begin{note} [Sufficiency of Symmetric Matrices] \label{N:sufSym}
If $M \in \mathbb{R}^{n \times n}$ and $V$ is orthogonal, the diagonals of $V^\top M V$ and $V^\top (M + M^\top) V$ differ only by a scalar multiple. Consequently, zeroing diagonal elements of $M$ and zeroing diagonal elements of $M + M^\top$ are equivalent, and it is sufficient to consider only symmetric matrices when orthogonally zeroing diagonal elements of real matrices.
\end{note}

In other words, orthogonally zeroing diagonal elements and thus, orthogonal hollowization, are symmetrization-invariant. Crucially, however, unitarily zeroing complex diagonal elements is \emph{not} necessarily Hermitization-invariant. For example, $M = \imath I$ is not hollow whereas $M + M^*$ is hollow.

Another useful note generalizes remarks made from \cite{tD20}.

\begin{note} [Invariance Across Translation of the Diagonal] \label{N:transDiag} 
Conjugation commutes with uniform translation of the diagonal; that is, $V^{-1} (M + cI) V = V^{-1} M V + cI$. Thus, finding a conjugation of $M$ where $k$ diagonal elements are $c$ is equivalent to zeroing the same $k$ diagonal elements of $M - c I$.
\end{note}

In particular, hollowization is equivalent to finding a conjugation with constant diagonal, so we use the terms ``hollowization'' and ``zeroing'' to refer to both throughout this paper.


Finally, since we consider unitary transformations for complex $M$ as well as orthogonal transformations for real $M$, we note a well-known result that smoothens transitions between the two.

\begin{note} [Equality of Unitary and Orthogonal Similarity Classes for Real Matrices] \label{N:realUnitary} 
Two real matrices are unitarily similar if and only if they are orthogonally similar. In particular, if ${M \in \mathbb{R}^{n \times n}}$ and $k$ diagonal elements of $M$ can be unitarily zeroed, then the same $k$ elements can be orthogonally zeroed. \cite{rH12,rM86}
\end{note}

\subsection{Nondefinite Operators} \label{Subsect:nondefOps} 

Orthogonally or unitarily zeroing diagonal elements appears in \cite{tD20,aN18,pF69} in the context of zeroing the diagonal of traceless operators, thus manifesting as hollowization. However, for reasons that will become clear, the most general context for zeroing diagonal elements is that of nondefinite operators.\footnote{\emph{Nondefinite} is to be distinguished from \emph{indefinite}, where an operator is indefinite if and only if it is not positive-semidefinite and not negative-semidefinite.}

\begin{definition} [Definite] \label{D:definite}
	An operator is {\bfseries definite} if and only if it is positive-definite or negative-definite. An operator is {\bfseries nondefinite} if and only if it is not definite.
\end{definition}

Informally speaking, Hermitian/symmetric definite matrices are associated with matrices that have relatively ``weighty'' diagonals. In numerical analysis and linear algebra, such matrices do not require pivoting, which is a blessing due to its associated high computational costs and tendency to destroy matrix structure. Moreover, Cholesky factorization, a stable factorization, is available for such matrices. \cite{gG96} Meanwhile, indefinite matrices are studied extensively in saddle-point problems. Not only are saddle-point matrices indefinite, but symmetric saddle-point matrices are typically highly indefinite as they tend to have many eigenvalues of both signs. \cite{mB05}

A plethora of characterizations of nondefinite operators are easily derived from characterizations of positive definite operators. \cite{rB09,rH12,gG96,aL05} Thus, the following equivalences, whose proofs are straightforward and omitted, are useful.

\begin{note} [Nondefinite Equivalences] \label{N:npdnndEquiv} 
Let $M$ be a linear operator. The following are equivalent.
\begin{enumerate}[(a)]
\item $M$ is nondefinite.
\item $M$ is indefinite or singular.
\item $M$ and $-M$ are not positive-definite.
\end{enumerate}
\end{note}

Since the Hermitian part of a complex matrix $M$ defines the same quadratic form $M$ does, a complex matrix inherits its definiteness from any Hermitization of it and similarly for real matrices and any corresponding symmetrization. \cite{gG96} Characterizations of positive definite Hermitian matrices in terms of their principal minors, such as Sylvester's criterion, entail a useful characterization of nondefinite matrices. \cite{rH12,rB09} 

\begin{note} [Principal Submatrix Characterization of Nondefinite Matrices] \label{N:prinSubNondef} 
Operator $M \in \mathbb{C}^{n \times n}$ is nondefinite if and only if some principal submatrix of some order of $M + M^*$ is nondefinite.
\end{note}

Notice Note \ref{N:prinSubNondef} implies $M$ is nondefinite if and only if $M + M^*$ is nondefinite. We can also see how Note \ref{N:prinSubNondef} implies a general abundance of nondefinite operators.

Given $M \in \mathbb{C}^{n \times n}$, recall the \emph{numerical range} of $M$ is $W(M)=\{ \boldsymbol{v}^* M \boldsymbol{v} : \boldsymbol{v} \in {\mathbb{C}}^{n}, \boldsymbol{v}^* \boldsymbol{v}=1\}$, the \emph{real numerical range} of $M$ is $W_\mathbb{R}(M)=\{ \boldsymbol{v}^\top M \boldsymbol{v} : \boldsymbol{v} \in {\mathbb{R}}^{n}, \boldsymbol{v}^\top \boldsymbol{v}=1\}$, and $M$ is Hermitian if and only if $W_\mathbb{R}(M) = W(M)$.

\begin{proposition} [Zeroing Diagonals and Nondefinite Equivalences] \label{P:zdndEquiv} 
Let $M \in \mathbb{C}^{n \times n}$. The following are equivalent.
\begin{enumerate}[(a)]
\item $M$ is nondefinite.
\item $0 \in W(M)$.
\item Any diagonal element of $M$ can be unitarily zeroed.

If $M$ is real, then any diagonal element of $M$ can be orthogonally zeroed.
\end{enumerate}
\end{proposition}

\begin{proof}
$(a) \Rightarrow (b)$ Since $M$ derives its nondefiniteness from its Hermitian part, we may assume $M$ is Hermitian. Thus, $W(M)$ is a real closed interval. By Note \ref{N:npdnndEquiv}$(b)$, $M$ is indefinite or singular. In either case, some eigenvalue of $M$ is nonpositive and some eigenvalue of $M$ is nonnegative. Since $W(M)$ is convex and contains all eigenvalues of $M$, $0 \in W(M)$.

$(b) \Rightarrow (c)$ In the unitary case, if $0 \in W(M)$ then $\boldsymbol{v}^* M \boldsymbol{v}=0$ for some unit $\boldsymbol{v} \in \mathbb{C}^{n}$. Then we can construct a unitary matrix $U$ where any $k$th column is $\boldsymbol{v}$, which implies the $k$th element of the diagonal of $U^* M U$ is 0. The orthogonal case follows from the same argument where $\boldsymbol{v}$ can chosen to be real in light of Note \ref{N:sufSym} and $W_\mathbb{R}(M) = W(M)$; alternatively, Note \ref{N:realUnitary} may be invoked.

$(c) \Rightarrow (a)$ If $(U^* M U)_{k,k} = 0$ for some unitary $U$, then $\boldsymbol{v}^* M \boldsymbol{v} = 0$, where $\boldsymbol{v}$ is the $k$th column of~U. This implies neither $M$ nor $-M$ is positive-definite, so $M$ must be nondefinite by Note \ref{N:npdnndEquiv}. The real/orthogonal case follows similarly.
\end{proof}

Proposition \ref{P:zdndEquiv} is to nondefinite operators what Theorem~1 of \cite{pF69} is to traceless operators; indeed, the former considers the diagonal when $0 \in W(M)$, the latter considers the diagonal when $\tr(M) \in W(M)$, and hollowization concerns the diagonal when $\tr(M) = 0$. Besides being useful for our ends, Proposition \ref{P:zdndEquiv} is important because the determination of necessary or sufficient conditions under which 0 is in the numerical range of an operator motivates much significant research. \cite{pP02} 

Operator-theoretic research analogous to Proposition \ref{P:zdndEquiv} for the Hardy space of the open unit disc is given in \cite{pB02} where conditions under which 0 is in the numerical range of a composition operator is derived.

Notice the proof of $(b) \Leftrightarrow (c)$ in Proposition \ref{P:zdndEquiv} works for any $z \in \mathbb{C}$ in place of 0 so that $z \in W(M)$ if and only if for all $k$, there exists unitary $U$ such that $(U^* M U)_{k,k} = z$. Indeed, a useful way to think of $W(M)$ for our purposes is as the set of all diagonal elements among all matrices unitarily similar to $M$, that is $W(M) = \{m \in \diag(M') : \forall \, M' \sim_U M\}$. Moreover, notice $z \in W(M)$ for all $M \in \mathbb{C}^{n \times n}$ if and only if $z = \frac{\tr M}{n}$. So the spectral mean $\frac{\tr M}{n}$ of $M$, which is also the barycenter of $W(M)$,\footnote{See \cite{sB07}.} is the unique $z \in \mathbb{C}$ that can be unitarily/orthogonally introduced to the diagonal of any complex/real matrix $M$. Far stronger conclusions are given for real $M$ and for Hermitian $M$ in Theorem \ref{T:conHol}.

\subsection{Considerations and Limitations on Conjugate Zeroing Diagonal Elements of Nondefinite Operators} \label{Subsect:limConjZer}

Proposition \ref{P:zdndEquiv} indicates a diagonal element of a nondefinite matrix can always be zeroed, but one may wonder if a diagonal element from each of two nondefinite matrices can be \emph{conjugate} zeroed. The answer is negative, and a counterexample for the orthogonal case with symmetric $L,M \in \mathbb{R}^{2 \times 2}$ is given by the pair 
\begin{equation}\label{E:CanonCounterexamp}
L= \begin{psmallmatrix}
1&0\\
0&-1\\
\end{psmallmatrix}
\quad
M = \begin{psmallmatrix}
0&1\\
1&0\\
\end{psmallmatrix}
\end{equation}
which is also the canonical counterexample (see \cite{tD20,lB61}) exhibiting not all pairs of real traceless matrices can be simultaneously orthogonally hollowized, so further serves as a counterexample to the claim that a diagonal element from each of any pair of real nondefinite matrices can be simultaneously orthogonally zeroed. 

One may wonder if (\ref{E:CanonCounterexamp}) serves as a counterexample by virtue of $L,M$ being restricted to two dimensions. This is not the case, and (\ref{E:nonzeroablePair1}) gives nondefinite symmetric $L,M \in \mathbb{R}^{3 \times 3}$ for which there is no orthogonal $R$ such that $R^\top L R$ and $R M R^\top$ both have a diagonal 0, and the same example gives nondefinite Hermitian $L,M \in \mathbb{C}^{3 \times 3}$ for which there is no unitary $U$ such that $U^* L U$ and $U M U^*$ both have a diagonal 0.
\begin{equation}\label{E:nonzeroablePair1}
L= \begin{psmallmatrix}
2&-1&-1\\
-1&1&0\\
-1&0&1
\end{psmallmatrix}
\quad
M = \begin{psmallmatrix}
1&0&0\\
0&1&1\\
0&1&1
\end{psmallmatrix}
\end{equation}
We can see no such $R$ exists by noticing in order for $R^\top L R$ to have a diagonal 0, all components of some column of $R$ must be equal, but in order for $R M R^\top$ to have a diagonal 0, some row of $R$ must be a permutation of $(0,-x,x)$. It is not possible for both such vectors to have unit length; thus, $R$ cannot be orthogonal. No such $U$ exists for the same reasons.

Furthermore, even if there exists an orthogonal $R$ such that $R^\top L R$ and $R M R^\top$ both have a diagonal 0, this does not imply for \emph{any} diagonal element of $R^\top L R$ and \emph{any} diagonal element of $R M R^\top$ there exists an orthogonal $R$ where both are 0. For instance, using an argument similar to that of the previous example, it is straightforward to show for 
\begin{equation}\label{E:nonzeroablePair2}
L= M = \begin{psmallmatrix}
1&0&0\\
0&1&1\\
0&1&1
\end{psmallmatrix},
\end{equation}
there exists an orthogonal $R$ such that $(R^\top L R)_{1,1} = (R M R^\top)_{i,i} = 0$ for $i = 1$ but not $i = 2,3$, and there exists a unitary $U$ such that $(U^* L U)_{1,1} = (U M U^*)_{i,i} = 0$ for $i = 1$ but not $i = 2,3$.\footnote{For the $i = 1$ case, a solution is $R = U = -\frac{1}{2}\begin{psmallmatrix}
0&\sqrt{2}&-\sqrt{2}\\
\sqrt{2}&1&1\\
-\sqrt{2}&1&1
\end{psmallmatrix}$.}

More general limitations exist even in the traceless case. In Note \ref{N:reqAss}$(a)$, we specify a class of traceless matrix pairs $L,M \in \mathbb{R}^{n \times n}$ for each $n \ge 2$ such that, given any combination of $l$ and $m \in \{m_1,m_2\}$ with $m_1 \neq m_2$, there is no orthogonal $R$ where $(R^\top L R)_{m_1,m_1} = (R^\top L R)_{m_2,m_2} = 0$ and $(R M R^\top)_{l,l} = 0$. Note \ref{N:sufSym} can be used to extend this construction to a symmetric class. In Note \ref{N:reqAss}$(b)$, we specify a class of traceless nonhermitian matrix pairs $L,M \in \mathbb{C}^{n \times n}$ for each $n \ge 2$ such that, given any combination of $l \in \{l_1,l_2\}$ with $l_1 \neq l_2$ and $m \in \{m_1,m_2\}$ with $m_1 \neq m_2$, there is no unitary $U$ such that $(U^* L U)_{m,m} = 0$ and $(U M U^*)_{l,l} = 0$.

\section{Hollowization and Traceless Operators} \label{Sect:holTraceless}

Hollowization is to traceless operators what zeroing a diagonal element is to nondefinite operators. All traceless operators are nondefinite, but being significantly more restrictive, their entire diagonal may be zeroed.

In \cite{pF69}, Fillmore investigates introducing 0s to the diagonal of a complex matrix $M$. He proves there exists a unitary $U$ such that $\diag(U^* M U) = (0,...,0, \tr M)$ if and only if ${\tr(M) \in W(M)}$, and, as a corollary, every complex traceless matrix is unitarily similar to a hollow matrix. The method of proof he employs is used in \cite{tD20} and is also used in our Lemma~\ref{L:conZer}, so we include it below.

\begin{theorem} [Fillmore's Theorem]\label{T:fillThm}
Operator $M \in \mathbb{C}^{n \times n}$ with $n \ge 2$ is unitarily similar to some~$M'$ with $\diag(M') = (0,...,0,\tr M)$ if and only if $\tr(M) \in W(M)$.
\end{theorem}

\begin{proof}
$(\Rightarrow)$ The conclusion from the fact that every diagonal element of a matrix must lie in its numerical range.

$(\Leftarrow)$ The assumption $\tr(M) \in W(M)$ implies $\boldsymbol{v}^* M \boldsymbol{v} = \tr(M)$ for some unit vector $\boldsymbol{v}$, so we may extend $\boldsymbol{v}$ to a unitary matrix $U$ whose first column is $\boldsymbol{v}$. Then $U^* M U = \begin{psmallmatrix} \tr M & \star \\ \star & N \end{psmallmatrix}$ for some ${N \in \mathbb{C}^{(n-1) \times (n-1)}}$, which implies $\tr(N) = 0$, so the spectral mean $\frac{\tr N}{n} = 0$. Since the spectral mean of any matrix lies inside its numerical range, $\tr(N) \in W(N)$.\footnote{See the discussions following Proposition \ref{P:zdndEquiv}.} Thus, we may continue with $N$ as we did with $M$ and proceed until the last diagonal element is 0. Finally, conjugating by a permutation matrix can move $\tr M$ anywhere along the diagonal.
\end{proof}

If $M$ is real, we may assume $M'$ is orthogonally similar to $M$ by Note \ref{N:realUnitary}. The orthogonal case takes a more prominent role in \cite{tD20} and in Section \ref{Sect:conjHol}, where significant differences from the general complex case begin to manifest.

\section{Conjugate Hollowization -- Zeroing Diagonals via $(U^* L U, U M U^*)$} \label{Sect:conjHol}

In \cite{tD20}, Damm and Fa\ss bender investigate introducing 0s to the diagonals of a pair of matrices $(L,M)$ via $(U^* L U, U^* M U)$ and prove, for real traceless $L,M$, there exists an orthogonal $R$ such that $R^\top L R$ is hollow and $R^\top M R$ is almost hollow. They also prove, for complex Hermitian traceless $L,M,N$, there exists a unitary $U$ such that $U^* L U$ is hollow, $U^* M U$ is hollow, and $U^* N U$ is almost hollow. These are instances of simultaneous hollowization and simultaneous (almost) hollowization. They also pose their Conjecture 22 that, for all real traceless $L,M$, there exists an orthogonal $R$ such that $R^\top L R$ is hollow and $R M R^\top$ is almost hollow.

In this section, we investigate introducing 0s to the diagonals of a pair $(L,M)$ via $(U^* L U, U M U^*)$. Let $L,M \in \mathbb{C}^{n \times n}$ with $n \ge 0$, $l = \frac{\tr L}{n} $, and $m= \frac{\tr M}{n}$. In Theorem \ref{T:conHol}, we prove $L,M$ are real implies there exists orthogonal $R$ such that $\diag(R^\top L\, R) = (l,...,l,l,l)$ and ${\diag(R \, M R^\top) = (m,...,m,*,*)}$, and $L,M$ are Hermitian implies there exists unitary $U$ such that $\diag(U^* L\, U) = (l,...,l)$ and $\diag(U M U^*) = (m,...,m)$. By Definition \ref{D:holDefs}, Definition \ref{D:twiHollowize}, and Note \ref{N:transDiag}, this is equivalent to finding an (almost) conjugate orthogonal hollowization for real $L,M$ and a conjugate unitary hollowization for Hermitian $L,M$. Thus, Theorem \ref{T:conHol} implies Conjecture 22 in \cite{tD20}. 

We first review some results regarding the joint numerical range and prove three lemmas for conjugate zeroing diagonal elements. With $L,M \in \mathbb{C}^{n \times n}$, recall the \emph{joint numerical range} of $(L,M)$ is $W(L,M)=\{ (\boldsymbol{v}^* L \boldsymbol{v}, \boldsymbol{v}^* M \boldsymbol{v}) : \boldsymbol{v} \in {\mathbb{C}}^{n}, \boldsymbol{v}^* \boldsymbol{v}=1\}$ and the \emph{real joint numerical range} of $(L,M)$ is $W_\mathbb{R}(L,M)=\{ (\boldsymbol{v}^\top L \boldsymbol{v}, \boldsymbol{v}^\top M \boldsymbol{v}) : \boldsymbol{v} \in {\mathbb{R}}^{n}, \boldsymbol{v}^\top \boldsymbol{v}=1\}$. For $L,M \in \mathbb{R}^{n \times n}$ with $n \ge 3$, $W_\mathbb{R}(L,M)$ is convex as proven by Brickman in \cite{lB61} using topological methods; another proof using topological methods is given in \cite{pB85}, a constructive proof is given in \cite{tD20}, and more elementary proofs appealing only to connectivity properties of quadrics in $\mathbb{R}^3$ are given in \cite{vY71,hP04,jM05}. If $L,M$ are also symmetric, then $W_\mathbb{R}(L,M) = W(L,M)$. \cite{lB61} For Hermitian $L,M \in \mathbb{C}^{n \times n}$ with $n \ge 0$, $W(L,M)$ is real and convex as proven in \cite{pB85}; closely related convexity results are given in \cite{cL00,eG04,pB91,cL20,yP94}. 

Notice it follows if either $L,M \in \mathbb{R}^{n \times n}$ or $L,M \in \mathbb{C}^{n \times n}$ are Hermitian with $n \ge 3$, then $W(L,M)$ is convex and $W_\mathbb{R}(L,M) = W(L,M)$. The condition $n \ge 3$ is required as $W_\mathbb{R}(L,M)$ may not even be convex for $n = 2$; for example, for $L,M$ from (\ref{E:CanonCounterexamp}), $W_\mathbb{R}(L,M)$ is the unit circle. \cite{lB61,tD20}

Lemma \ref{L:jointSpecMeans} is a generalization of Proposition~5 from \cite{tD20}, where a short proof invoking Brickman's theorem as well as a novel constructive proof is given for the case $\tr A = \tr B = 0$. A constructive proof of Brickman's theorem is also given in \cite{tD20}. This suggests an additional, constructive proof for our lemma.

\begin{lemma} [Conditions for $\frac{1}{n}(\tr A,\tr B) \in W(A,B)$]\label{L:jointSpecMeans}
For all $A,B \in \mathbb{R}^{n \times n}$ with $n \ge 3$ and for all Hermitian $A,B \in \mathbb{C}^{n \times n}$ with $n \ge 0$, there exists unit $\boldsymbol{v} \in \mathbb{R}^n$ such that $\boldsymbol{v}^\top A \boldsymbol{v} = \frac{\tr A}{n}$ and $\boldsymbol{v}^\top B \boldsymbol{v} = \frac{\tr B}{n}$.
\end{lemma}

\begin{proof}
For $A,B \in \mathbb{R}^{n \times n}$ with $n \ge 3$, subtracting out the trace of $A$, using Note 3.2, and using the fact every real traceless matrix can be orthogonally hollowized, we may assume $A$ is hollow without loss of generality. Note there exist $b_1,b_2 \in \diag(B)$ such that $b_1 \le \frac{\tr B}{n}$ and $b_2 \ge \frac{\tr B}{n}$. However, this implies there exist standard basis vectors $\boldsymbol{u}_1,\boldsymbol{u}_2$ such that $(\boldsymbol{u}_1^\top A \boldsymbol{u}_1,\boldsymbol{u}_1^\top B \boldsymbol{u}_1) = (0,b_1)$ and $(\boldsymbol{u}_2^\top A \boldsymbol{u}_2,\boldsymbol{u}_2^\top B \boldsymbol{u}_2) = (0,b_2)$. Since $W_\mathbb{R}(A,B)$ is convex, there exists a unit $\boldsymbol{v}$ where $\boldsymbol{v}^\top A \boldsymbol{v} = 0 = \frac{\tr A}{n}$ and $\boldsymbol{v}^\top B \boldsymbol{v} = \frac{\tr B}{n}$.

For Hermitian $A,B \in \mathbb{C}^{n \times n}$ with $n \ge 0$, the proof is similar to that of $(a)$ as every complex traceless matrix can be unitarily hollowized, Hermitian matrices have real diagonals, and $W(L,M)$ is real and convex.
\end{proof}

Lemma \ref{L:jointSpecMeans} provides useful conditions determining when the joint numerical range of two matrices contains the ordered pair of their spectral means, and the spectral mean $\frac{\tr M}{n}$ of a matrix $M$ of size $n$ is the barycenter of $W(M)$. Thus, Lemma \ref{L:jointSpecMeans} generalizes commentary following Proposition~\ref{P:zdndEquiv}. In particular, we may conclude $z \in W(A,B)$ for all $A,B \in \mathbb{R}^{n \times n}$ with $n \ge 3$ and for all Hermitian $A,B \in \mathbb{C}^{n \times n}$ with $n \ge 0$ if and only if $z = \frac{1}{n}(\tr A,\tr B)$, making $\frac{1}{n}(\tr A,\tr B)$ the unique point whose components can be simultaneously introduced to the diagonals of any pair of real or complex Hermitian matrices.

Lemma \ref{L:hHComp} solves a simple matrix completion problem from any real unit vector into a real Householder matrix. We follow \cite{gG96} for conventions regarding Householder matrices and vectors.

\begin{lemma} [Householder Completion] \label{L:hHComp}
Let $n \ge 2$. For any unit $\boldsymbol{v} \in \mathbb{R}^n$,  there exists a real Householder matrix $H = H(\boldsymbol{v})$ whose first column is $\boldsymbol{v}$.

As an explicit construction, define vector $\boldsymbol{h} = (h_1,...,h_n)$ such that
\begin{equation}\label{E:hHComp}
h_1 = \pm\frac{1}{\sqrt{2}}\sqrt{1-v_1}, \quad
h_k = -\frac{1}{2} \frac{v_k}{h_1} \text{ for } k \neq 1
\end{equation}
in the case $v_1 \neq 1$, and $\boldsymbol{h}$ to be any standard basis vector or its opposite with $h_1 = 0$ in the case $v_1 = 1$. Vector $\boldsymbol{h}$ is a real unit $n$-vector, and the Householder matrix with Householder vector $\boldsymbol{h}$ has the properties of $H$.
\end{lemma}

\begin{proof}
The proof is given by the construction. It is straightforward to algebraically verify $\boldsymbol{h}$ is a real unit $n$-vector, then the canonical equation $H = I - 2 \boldsymbol{h} \boldsymbol{h}^*$ specifies $H$. \cite{gG96} It is straightforward to algebraically verify the first column of $H$ is $\boldsymbol{v}$.
\end{proof}

Crucially, recall real Householder matrices are symmetric and orthogonal. Lemma \ref{L:hHComp} implies we may rely on Householder transformations, which have very nice properties, for transforming diagonal elements.

We may combine Lemma \ref{L:jointSpecMeans} and Lemma \ref{L:hHComp} to build up a conjugate zeroing scheme.

\begin{lemma}[Conjugate Zeroing]\label{L:conZer}
\,
\begin{enumerate}[(a)]
\item For all $L,M \in \mathbb{R}^{n \times n}$ with $n \ge 3$, there exists orthogonal $R$ such that $\diag(R^\top L\, R) = (l,l,...,l)$ and $\diag(R \, M \,R^\top) = (m,*,...,*)$ for $l = \frac{\tr L}{n} $ and $m= \frac{\tr M}{n}$.

\item For all Hermitian $L,M \in \mathbb{C}^{n \times n}$ with $n \ge 0$, there exists unitary $U$ such that $\diag(U^* L\, U) = (l,l,...,l)$ and $\diag(U \, M \,U^*) = (m,*,...,*)$ for $l = \frac{\tr L}{n} $ and $m= \frac{\tr M}{n}$.
\end{enumerate}
\end{lemma}

\begin{proof}
$(a)$ As in the proof of Theorem \ref{T:fillThm} and proofs of claims made in \cite{pF69,tD20}, the canonical way  to construct an orthogonal $R \in \mathbb{R}^{n \times n}$ such that $\diag(R^\top L\, R) = (l,...,l,l,l)$ is to construct a sequence of orthogonal $R_i$ such that $R = R_1...R_{n-1}$ where
\begin{align}
\begin{split}
    &\diag(R_1^\top L\, R_1) = (l,*,...,*)  \\
    &\diag(R_2^\top R_1^\top L\, R_1 R_2) = (l,l,*,...,*)    \\
    &\hspace{40pt} \vdots    \\
    &\diag(R_{n-1}^\top ... R_1^\top L\, R_1 ... R_{n-1}) = \diag(R^\top L\, R) = (l,...,l).
\end{split}
\end{align}
For $n \ge 3$, using Lemma \ref{L:jointSpecMeans}, we may replace $R_{n-1}$ with $H(\boldsymbol{v})$ from Lemma \ref{L:hHComp} where $\boldsymbol{v}$ satisfies $\boldsymbol{v}^\top (R_{n-2}^\top ... R_1^\top L\, R_1 ... R_{n-2}) \boldsymbol{v} = l$ and $\boldsymbol{v}^\top M \boldsymbol{v} = m$.

$(b)$ The proof is similar to that of $(a)$ with unitary $U_i$ in place of orthogonal $R_i$ and where $n \ge 3$ can be relaxed to $n \ge 0$ using Lemma \ref{L:jointSpecMeans}.
\end{proof}

Notice Lemma \ref{L:conZer} already generalizes Theorem \ref{T:fillThm}. However, Theorem \ref{T:conHol} takes us much further.

We are now in the position to prove our main result, Theorem~\ref{T:conHol}. We prove Theorem~\ref{T:conHol}$(a)$ by induction on size $n$ and then show Theorem \ref{T:conHol}$(b)$ has a similar proof.

\begin{theorem}[Conjugate Hollowization]\label{T:conHol}
Let $n \ge 0$.
\begin{enumerate}[(a)]
\item For all $L,M \in \mathbb{R}^{n \times n}$, there exists orthogonal $R$ such that $\diag(R^\top L\, R) = (l,...,l,l,l)$ and $\diag(R \, M \,R^\top) = (m,...,m,*,*)$ for $l = \frac{\tr L}{n} $ and $m= \frac{\tr M}{n}$.

\item For all Hermitian $L,M \in \mathbb{C}^{n \times n}$, there exists unitary $U$ such that $\diag(U^* L\, U) = (l,...,l)$ and $\diag(U M U^*) = (m,...,m)$ for $l = \frac{\tr L}{n}$ and $m= \frac{\tr M}{n}$.
\end{enumerate}
\end{theorem}

\begin{proof}
$(a)$ The theorem is tautological for $n = 0,1$, so assume $n \ge 2$. Using Lemmas \ref{L:jointSpecMeans} and \ref{L:hHComp}, let $\psi_{A,B}(\_) = H_{A,B} \, \_ \, H_{A,B}$ for $A,B \in \mathbb{R}^{n \times n}$ denote conjugation by $H_{A,B}$ where $H_{A,B}$ is a Householder matrix whose first column is $\boldsymbol{v}_{A,B}$ such that $\boldsymbol{v}_{A,B}^\top A \boldsymbol{v}_{A,B} = \frac{\tr A}{n}$ and $\boldsymbol{v}_{A,B}^\top B \boldsymbol{v}_{A,B} = \frac{\tr B}{n}$.

Let $\phi$ denote conjugation by $R$ where $\phi(\_)=R^\top \_\, R$ and $\phi^\top(\_)=R \, \_ \,R^\top$, and let $\phi_k$ with $3 \leq k \leq n$ be this conjugation action on the $k \times k$ lower diagonal block and the identity elsewhere so the last $k$ elements of $\diag(\phi_k(\_))$ are equal and the last $k$ elements of $\diag(\phi_k^\top(\_))$ are equal except for, at most, the last two. Notice $\phi_n = \phi$. We will prove $\phi$ exists via induction on size $n$.

\noindent {\bfseries Base Case:} $n = 3$

By Lemma \ref{L:conZer}, there exists $\phi$ such that $\diag(\phi(L)) = (l,l,l)$ and $\diag(\phi^\top(M)) = (m,*,*)$.

\noindent {\bfseries Induction Hypothesis}

For all $L,M \in \mathbb{R}^{(n-1) \times (n-1)}$, there exists $\phi$ such that $\diag(\phi(L)) = (l,...,l,l,l)$ and $\diag(\phi^\top(M)) = (m,...,m,*,*)$.

\noindent {\bfseries Induction Consequent}

We now prove $\phi$ exists for all $L,M \in \mathbb{R}^{n \times n}$. The key is to consider $\phi = \phi_{n-1} \, \circ \, \psi_{A,B} \, \circ \, \phi_{n-1}$ where $A = \phi_{n-1}(L)$ and $B = \phi_{n-1}^\top(M)$. Using a hat $\, \hat{} \,$ as a label to distinguish applications of $\phi_{n-1}$, by the induction hypothesis, there exist $\phi_{n-1}, \hat{\phi}_{n-1}$ such that
\begin{flalign}
\hat{\phi}_{n-1} \, \circ \, \psi_{A,B} \, \circ \, \phi_{n-1}(L) 
&= \hat{\phi}_{n-1} \, \circ \, \psi_{A,B}(\begin{pNiceMatrix}[small,xdots/shorten=7pt,xdots/radius=0.7pt,margin=0.1em] 
*&&&&&\\
&a&&&&\\
&&\phantom{a}&&&\\
&&&a&&\\
&&&&a&\\
&&&&&a
    \CodeAfter
    \line{2-2}{4-4}
\end{pNiceMatrix})
= \hat{\phi}_{n-1}(\begin{pNiceMatrix}[small,xdots/shorten=7pt,xdots/radius=0.7pt,margin=0.1em] 
l&&&&&\\
&*&&&&\\
&&\phantom{*}&&&\\
&&&*&&\\
&&&&*&\\
&&&&&*
    \CodeAfter
    \line{2-2}{4-4}
\end{pNiceMatrix})
=\mathrlap{\begin{pNiceMatrix}[small,xdots/shorten=10pt,xdots/radius=0.7pt,margin,margin=0.2em] 
\hspace{.08em}l\hspace{.08em}&&&&&\\
&\phantom{\hspace{.08em}l\hspace{.08em}}&&&&\\
&&\phantom{\hspace{.08em}l\hspace{.08em}}&&&\\
&&&\hspace{.08em}l\hspace{.08em}&&\\
&&&&\hspace{.08em}l\hspace{.08em}&\\
&&&&&\hspace{.08em}l\hspace{.08em}
    \CodeAfter
    \line{1-1}{4-4}
\end{pNiceMatrix}}&&\label{E:transL}
\\
\hat{\phi}_{n-1}^\top \, \circ \, \psi_{A,B} \, \circ \, \phi_{n-1}^\top(M) 
&= \hat{\phi}_{n-1}^\top \, \circ \, \psi_{A,B}(\begin{pNiceMatrix}[small,xdots/shorten=7pt,xdots/radius=0.7pt,margin=0.1em]
*&&&&&\\
&b&&&&\\
&&\phantom{b}&&&\\
&&&b&&\\
&&&&*&\\
&&&&&*
    \CodeAfter
    \line{2-2}{4-4}
\end{pNiceMatrix})
= \hat{\phi}_{n-1}^\top(\begin{pNiceMatrix}[small,xdots/shorten=7pt,xdots/radius=0.7pt,margin=0.1em]
m&&&&&\\
&*&&&&\\
&&\phantom{*}&&&\\
&&&*&&\\
&&&&*&\\
&&&&&*
    \CodeAfter
    \line{2-2}{4-4}
\end{pNiceMatrix})
=\mathrlap{\begin{pNiceMatrix}[small,xdots/shorten=10pt,xdots/radius=0.7pt,margin,margin=0.1em]
m&&&&&\\
&\phantom{l}&&&&\\
&&\phantom{l}&&&\\
&&&m&&\\
&&&&*&\\
&&&&&*
    \CodeAfter
    \line{1-1}{4-4}
\end{pNiceMatrix}}&&\label{E:transM}
\end{flalign}
for some constants $a$ and $b$.

Let $M_{n-1}$ and $\hat{M}_{n-1}$ be the $(n-1) \times (n-1)$ lower diagonal blocks of $M$ and $\psi_{A,B} \, \circ \, \phi_{n-1}^\top(M)$, respectively. Eqs. (\ref{E:transL}) and (\ref{E:transM}) imply, for every $M \in R^{n \times n}$, so for every $M_{n-1} \in R^{(n-1) \times (n-1)}$, there exists an induced $\hat{M}_{n-1}$. But since $\hat{M}_{n-1} \in R^{(n-1) \times (n-1)}$, we can swap $\hat{M}_{n-1} \leftrightarrow M_{n-1}$. That is, if~(\ref{E:transM}) exists for every $M$, then (\ref{E:transM}) exists for every $M$ with $\hat{M}_{n-1}$ as its $(n-1) \times (n-1)$ lower diagonal block. Thus, there exist $\phi_{n-1}, \hat{\phi}_{n-1}$ satisfying (\ref{E:transL}) and
\begin{equation}
\phi_{n-1}^\top \, \circ \, \psi_{A,B} \, \circ \, \hat{\phi}_{n-1}^\top(M) = ... =\begin{pNiceMatrix}[small,xdots/shorten=10pt,xdots/radius=0.7pt,margin=0.2em]
m&&&&&\\
&\phantom{l}&&&&\\
&&\phantom{l}&&&\\
&&&m&&\\
&&&&*&\\
&&&&&*
    \CodeAfter
    \line{1-1}{4-4}
\end{pNiceMatrix}\label{E:transM2}
\end{equation}
(notice the position of the hat $\, \hat{} \,$ between (\ref{E:transM}) and (\ref{E:transM2})). But (\ref{E:transM2}) is the transpose of (\ref{E:transL}). That is, letting (\ref{E:transL}) be $\phi$, (\ref{E:transM2}) is $\phi^\top$, so we have proven the desired $\phi$ exists.

$(b)$ The proof is similar to that of $(a)$ but with complex unitary transformations in place of real orthogonal ones. The theorem is tautological for $n = 0,1$, so assume $n \ge 2$.

Following the proof of $(a)$, let $\psi_{A,B}(\_) = H_{A,B} \, \_ \, H_{A,B}$ for Hermitian $A,B \in \mathbb{C}^{n \times n}$. Let $\phi(\_)=U^* \_\, U$ and $\phi^*(\_)=U \, \_ \,U^*$, and let $\phi_k$ with $0 \leq k \leq n$ be this conjugation action on the $k \times k$ lower diagonal block and the identity elsewhere. We will again prove $\phi$ exists via induction on $n$.

For the base case $n = 2$, by Lemma \ref{L:conZer}, there exists $\phi$ such that $\diag(\phi(L)) = (l,l)$ and $\diag(\phi^*(M)) = (m,m)$. The induction hypothesis is, for all Hermitian $L,M \in \mathbb{C}^{(n-1) \times (n-1)}$, there exists $\phi$ such that $\diag(\phi(L)) = (l,...,l)$ and $\diag(\phi^*(M)) = (m,...,m)$.

For the induction consequent, we now prove $\phi$ exists for all Hermitian $L,M \in \mathbb{C}^{n \times n}$ by considering $\phi = \phi_{n-1} \, \circ \, \psi_{A,B} \, \circ \, \phi_{n-1}$ where $A = \phi_{n-1}(L)$ and $B = \phi_{n-1}^*(M)$. Again using a hat $\, \hat{} \,$ to distinguish applications of $\phi_{n-1}$, there exist $\phi_{n-1}, \hat{\phi}_{n-1}$ such that
\begin{align}
\hat{\phi}_{n-1} \, \circ \, \psi_{A,B} \, \circ \, \phi_{n-1}(L) 
&= \hat{\phi}_{n-1} \, \circ \, \psi_{A,B}(\begin{pNiceMatrix}[small,xdots/shorten=7pt,xdots/radius=0.7pt,margin=0.2em] 
*&&&\\
&a&&\\
&&\phantom{a}&\\
&&&a
    \CodeAfter
    \line{2-2}{4-4}
\end{pNiceMatrix})
= \hat{\phi}_{n-1}(\begin{pNiceMatrix}[small,xdots/shorten=7pt,xdots/radius=0.7pt,margin=0.2em] 
l&&&\\
&*&&\\
&&\phantom{*}&\\
&&&*
    \CodeAfter
    \line{2-2}{4-4}
\end{pNiceMatrix})
=\begin{pNiceMatrix}[small,xdots/shorten=7pt,xdots/radius=0.7pt,margin=0.2em] 
\hspace{.08em}l\hspace{.08em}&&&\\
&\hspace{.08em}l\hspace{.08em}&&\\
&&\phantom{\hspace{.08em}l\hspace{.08em}}&\\
&&&\hspace{.08em}l\hspace{.08em}
    \CodeAfter
    \line{2-2}{4-4}
\end{pNiceMatrix}\label{E:transLComp}
\\
\hat{\phi}_{n-1}^* \, \circ \, \psi_{A,B} \, \circ \, \phi_{n-1}^*(M) 
&= \hat{\phi}_{n-1}^* \, \circ \, \psi_{A,B}(\begin{pNiceMatrix}[small,xdots/shorten=7pt,xdots/radius=0.7pt,margin=0.2em] 
*&&&\\
&b&&\\
&&\phantom{b}&\\
&&&b
    \CodeAfter
    \line{2-2}{4-4}
\end{pNiceMatrix})
= \hat{\phi}_{n-1}^*(\begin{pNiceMatrix}[small,xdots/shorten=7pt,xdots/radius=0.7pt,margin=0.2em] 
m&&&\\
&*&&\\
&&\phantom{*}&\\
&&&*
    \CodeAfter
    \line{2-2}{4-4}
\end{pNiceMatrix})
=\begin{pNiceMatrix}[small,xdots/shorten=10pt,xdots/radius=0.7pt,margin=0.2em] 
m&&&\\
&m&&\\
&&\phantom{m}&\\
&&&m
    \CodeAfter
    \line{2-2}{4-4}
\end{pNiceMatrix}\label{E:transMComp}
\end{align}
for some constants $a$ and $b$. This implies, for the same reason as in the proof of $(a)$, there exist $\phi_{n-1}, \hat{\phi}_{n-1}$ satisfying (\ref{E:transLComp}) and
\begin{equation}
\phi_{n-1}^* \, \circ \, \psi_{A,B} \, \circ \, \hat{\phi}_{n-1}^*(M) = ... =\begin{pNiceMatrix}[small,xdots/shorten=10pt,xdots/radius=0.7pt,margin=0.2em] 
m&&&\\
&m&&\\
&&\phantom{m}&\\
&&&m
    \CodeAfter
    \line{2-2}{4-4}
\end{pNiceMatrix}\label{E:transMComp2}.
\end{equation}
But (\ref{E:transMComp2}) is the conjugate transpose of (\ref{E:transLComp}), so (\ref{E:transLComp}) specifies the desired $\phi$.
\end{proof}

Notice for traceless $L,M$, Theorem \ref{T:conHol} implies Conjecture 22 from \cite{tD20} and an equivalent statement for the complex Hermitian case. 

\begin{corollary} [Conjecture 22 from \cite{tD20} and Conjugate Hollowization] \label{C:conject22}
Let $n \ge 0$.
\begin{enumerate}[(a)]
\item For all traceless $L,M \in \mathbb{R}^{n \times n}$, there exists orthogonal $R$ such that $R^\top L\, R$ is hollow and $R \, M \,R^\top$ is almost hollow.

\item For all traceless Hermitian $L,M \in \mathbb{C}^{n \times n}$, there exists unitary $U$ such that $U^* L\, U$ and $U M U^*$ are both hollow.
\end{enumerate}
\end{corollary}

\begin{proof}
Set $\tr(L) = \tr(M) = 0$ in Theorem \ref{T:conHol} and notice $l = m = 0$.
\end{proof}

Since traceless matrices have the trace of smallest magnitude possible, we expect these matrices to be roughly ``the most nondefinite'' matrices and in a sense, the ideal nondefinite matrices. Indeed, we know from Theorem \ref{T:fillThm} and claims from \cite{pF69,tD20} that a real matrix is traceless if and only if it is orthogonally hollowizable, and a complex matrix is traceless if and only if it is unitarily hollowizable. Theorem \ref{T:conHol} implies these claims can be strengthened -- a real matrix is traceless if and only if it is orthogonally hollowizable by $R$, where $R^\top$ is used to orthogonally transform all but, at most, two diagonal elements of another matrix, and a complex Hermitian matrix is traceless if and only if it is unitarily hollowizable by $U$, where $U^*$ is used to unitarily transform all diagonal elements of another matrix. 

As Note \ref{N:reqAss} shows, whereas real $L,M$ can only be conjugate zeroed into an almost hollow pair in general, complex Hermitian $L,M$ can be fully conjugate hollowized, but the Hermitian assumption cannot be weakened in general.

\begin{note} [Required Assumptions for Conjugate Hollowization] \label{N:reqAss}
\,
\begin{enumerate}[(a)]
\item That all but at most two elements of $\diag(R M R^\top)$ are $m$ in Theorem \ref{T:conHol}$(a)$ is best-possible. For instance, there exist real symmetric traceless matrices $L,M$ for which there does not exist an orthogonal $R$ such that both $R^\top L \, R$ and $R M R^\top$ are hollow. An example is (\ref{E:CanonCounterexamp}), but we can generalize extensively. We specify a class of traceless matrix pairs $L,M \in \mathbb{R}^{n \times n}$ for each $n \ge 2$ such that, given any combination of $l$ and $m \in \{m_1,m_2\}$ with $m_1 \neq m_2$, there is no orthogonal $R$ such that $(R^\top L R)_{m_1,m_1} = (R^\top L R)_{m_2,m_2} = 0$ and $(R M R^\top)_{l,l} = 0$.

Let traceless $L,M \in \mathbb{R}^{n \times n}$ with $n \ge 2$ such that $L$ is the diagonal matrix whose diagonal is~1 in every position except $L_{l,l} = -(n-1)$, and $M$ is 0 everywhere except $M_{m_1,m_2} = 1$ for $m_1 \neq m_2$. For any $k$, it is straightforward to show $(R^\top L R)_{k,k} = 0$ implies $R_{1,k}^2 + ... + R_{n,k}^2 - n R_{l,k}^2 = 0$, and $R$ is orthogonal implies $R_{1,k}^2 + ... + R_{n,k}^2 = 1$. Together, these conditions imply, for all~$k$, $R_{l,k}^2 = \frac{1}{n}$. Thus, $(R^\top L R)_{m_1,m_1} = (R^\top L R)_{m_2,m_2} = 0$ implies $R_{l,m_1}^2 = R_{l,m_2}^2 = \frac{1}{n}$. However, this implies $R_{m_1,m_1} \neq 0$ and $R_{m_2,m_2} \neq 0$, so $(R M R^\top)_{l,l} \neq 0$.

Moreover, it is straightforward to extend this construction to a class of real symmetric traceless matrices using Note \ref{N:sufSym}.

\item The assumption $L,M$ are Hermitian in Theorem \ref{T:conHol}$(b)$ is necessary. We specify a class of traceless matrix pairs $L,M \in \mathbb{C}^{n \times n}$ for each $n \ge 2$ such that, given any combination of $l \in \{l_1,l_2\}$ with $l_1 \neq l_2$ and $m \in \{m_1,m_2\}$ with $m_1 \neq m_2$, there is no unitary $U$ such that $(U^* L U)_{m,m} = 0$ and $(U M U^*)_{l,l} = 0$. This construction is a modification and generalization of the construction in Remark 16 of \cite{tD20} specifying nonhermitian traceless $L,M \in \mathbb{C}^{n \times n}$ for each $n \ge 2$ such that $W(L,M)$ is not convex.

Let traceless $L,M \in \mathbb{C}^{n \times n}$ with $n \ge 2$ such that $L$ is the diagonal matrix whose diagonal is $\imath$ in every position except $L_{l_1,l_1} = 1$ and $L_{l_2,l_2} = -1-(n-2) \imath$ for $l_1 \neq l_2$, and $M$ is 0 everywhere except $M_{m_1,m_2} = 1$ for $m_1 \neq m_2$. It is straightforward to confirm, for all $U \in \mathbb{C}^{n \times n}$, if $(U^* L U)_{m,m} = 0$ and ($U_{l_1,m} = 0$ or $U_{l_2,m} = 0$), then column $U_{*,j}$ is the 0 vector, which precludes $U$ from being unitary. However, $(U M U^*)_{l,l} = 0$ implies $U_{l,m_1} = 0$ or $U_{l,m_2} = 0$. Thus, there is no unitary $U$ such that $(U^* L U)_{m,m} = 0$ and $(U M U^*)_{l,l} = 0$ for $l \in \{l_1,l_2\}$ and $m \in \{m_1,m_2\}$.
\end{enumerate}
\end{note}

The limitations mentioned in Note \ref{N:reqAss} are ultimately derived from the differences in the convexity of the joint numerical range when $L,M$ are real, Hermitian, or neither. In particular, let $L,M \in \mathbb{C}^{n \times n}$. If $L,M$ are real, then $W(L,M)$ is guaranteed to be convex only for all $n \ge 3$. If $L,M$ are Hermitian, the $W(L,M)$ is convex for $n \ge 0$. If $L,M$ are neither real nor Hermitian, then there is no $n \ge 2$ for which $W(L,M)$ is guaranteed to be convex, and an example for every $n \ge 2$ is given in Remark 16 of \cite{tD20}.

By setting $L = M$, we can derive revealing statements about zeroing the diagonal of a single operator and in particular, statements measuring the freedom and constraint present. In particular, every real matrix is orthogonally similar by some $R$ to a matrix whose diagonal is constant and is orthogonally similar by $R^\top$ to a matrix whose diagonal is constant except, at most, in its last two positions. Moreover, every complex Hermitian matrix is unitarily similar by some $U$ as well as by $U^*$ to matrices whose diagonals are constant.

\begin{corollary} [Conjugate Hollowization of a Single Matrix] \label{C:conHolOne} 
Let $n \ge 0$.
\begin{enumerate}[(a)]
\item For all $M \in \mathbb{R}^{n \times n}$, there exists orthogonal $R$ such that $\diag(R^\top M\, R) = (m,...,m,m,m)$ and $\diag(R \, M \,R^\top) = (m,...,m,*,*)$ for $m= \frac{\tr M}{n}$.

In particular, a real matrix $M$ is traceless if and only if there exists orthogonal $R$ such that $R^\top M R$ is hollow and $R \, M R^\top$ is almost hollow.
\item For all Hermitian $M \in \mathbb{C}^{n \times n}$, there exists unitary $U$ such that $\diag(U^* M \, U) = \diag(U M U^*) = (m,...,m)$ for $m= \frac{\tr M}{n}$.

In particular, a complex Hermitian matrix $M$ is traceless if and only if there exists unitary~$U$ such that $U^* M U$ and $U M U^*$ are both hollow.
\end{enumerate}
\end{corollary}

\begin{proof}
Let $L = M$ in Theorem \ref{T:conHol} and notice $m = 0$ when $M$ is traceless.
\end{proof}

Corollary \ref{C:conHolOne} also strengthens the results from \cite{pF69,tD20} that a real matrix is traceless if and only if it is orthogonally hollowizable, and a complex Hermitian matrix is traceless if and only if it is unitarily hollowizable. In particular, Corollary \ref{C:conHolOne} implies a real matrix is traceless if and only if it is orthogonally hollowizable by some $R$ and orthogonally almost hollowizable by $R^\top$, and a complex Hermitian matrix is traceless if and only if it is unitarily hollowizable by some $U$ as well as by $U^*$.

\section{Conclusion and Outlook} \label{Sect:conc}

Our research determines conditions under which $U^{-1} M \, U$ has 0s on its diagonal, where $U$ is orthogonal for real $M$ and $U$ is unitary for complex $M$. Then, for a pair $L,M$, we determine conditions under which $U^{-1} L \, U$ and $U M U^{-1}$ have 0s on their diagonals, where $U$ is orthogonal for real $L,M$ and $U$ is unitary for complex Hermitian $L,M$.

Our primary contribution is Theorem~\ref{T:conHol} where, for $L,M \in \mathbb{C}^{n \times n}$, $l = \frac{\tr L}{n}$, and $m= \frac{\tr M}{n}$, we prove if $L,M$ are real, then there exists orthogonal $R$ such that $\diag(R^\top L\, R) = (l,...,l,l,l)$ and $\diag(R \, M \,R^\top) = (m,...,m,*,*)$, and if $L,M$ are Hermitian, then there exists unitary $U$ such that $\diag(U^* L\, U) = (l,...,l)$ and $\diag(U M U^*) = (m,...,m)$. 

This implies Corollary \ref{C:conject22} -- for traceless $L,M \in \mathbb{C}^{n \times n}$, if $L,M$ are real, then there exists orthogonal $R$ such that $\diag(R^\top L\, R) = (0,...,0,0,0)$ and $\diag(R \, M \,R^\top) = (0,...,0,*,*)$, and if $L,M$ are Hermitian, then there exists unitary $U$ such that $\diag(U^* L\, U) = \diag(U M U^*) = (0,...,0)$. The real/orthogonal case is Conjecture 22 from \cite{tD20}.

Theorem~\ref{T:conHol} also implies Corollary \ref{C:conHolOne}, which reveals the implications of Theorem~\ref{T:conHol} for a single matrix in both the real case and the complex Hermitian case. This leads to novel characterizations of traceless matrices and stronger forms of Fillmore's theorems from \cite{pF69,tD20}.

We have shown nondefinite operators are a more general context for introducing 0s to diagonals than traceless operators are. Yet, there is little further literature on this. For example, given $M$ of size $n$, we have shown nondefiniteness is necessary and sufficient to introduce one 0, \cite{pF69} provides necessary and sufficient conditions under which $n-1$ diagonal elements may be zeroed, and we know tracelessness is necessary and sufficient for all diagonal elements to be zeroed. However, are there conditions on $M$ under which $k$ diagonal elements can be orthogonally/unitarily zeroed for $2 \leq k \leq n-2$?

We also propose a conjecture we believe to be true based on the truth of our Corollary \ref{C:conject22} and Proposition 13 from \cite{tD20}.

\begin{conjecture} [Extension of Conjugate Hollowization to Three Hermitian Operators] \label{Con:extConjHol}
For all Hermitian traceless $L,M,N \in \mathbb{C}^{n \times n}$ with $n \ge 2$, there exists unitary~$U$ such that $U^* L\, U$ and $U M U^*$ are hollow, and $U^* N\, U$ is almost hollow.
\end{conjecture}

Despite its similarity to simultaneous zeroing/hollowization, conjugate zeroing/hollowization seems to require some different approaches in proof. This is illustrated by the differences between our proof of Theorem~\ref{T:conHol} and the proofs of Proposition 5 and Proposition 13 in \cite{tD20}. Conjecture~\ref{Con:extConjHol} is interesting because a proof of it may require an approach different from both what is present in this paper as well as from \cite{tD20}.

\section{Acknowledgements} \label{Sect:Acknowledgements}

Foremost gratitude is reserved for Tobias Damm and Heike Fa\ss bender, as the entirety of this work was inspired by \cite{tD20}. I would also like to extend my gratitude to \emph{The Symposium: Philosophy Community of Chicago} as well as to Edward Mogul (Loyola University) and Stephen Walker (University of Chicago) for inspiring thoughts that fluttered in the vicinity of this work. I thank the staff of \emph{The Violet Hour} for hosting the late-night musings that fed this work. Finally, the love and support of my family is the ground I stand upon.

\begin{center}
\emph{Never trust a skinny chef.}
\end{center}



\printbibliography

\end{document}